\RequirePackage{fix-cm}
\documentclass[smallcondensed]{svjour3}     
\usepackage{stackrel}
\usepackage[shortlabels]{enumitem}
\usepackage{epsfig}
\usepackage[final]{pdfpages}
\usepackage{float}
\usepackage{amsfonts}
\usepackage{amsmath}
\usepackage{amssymb}

\smartqed  
\usepackage{graphicx}

\newcommand{\fft}{\mathcal{F}}
\newcommand{\paren}[1]{\left(#1\right)}

\newcommand{\parennn}[1]{\left\{#1\right\}}

\newcommand{\norm}[1]{\left\|#1\right\|}
\renewcommand{\Bbb}{\mathbb{B}}
\newcommand{\Cbb}{\mathbb{C}}

\newcommand{\Nbb}{\mathbb{N}}
\newcommand{\Rbb}{\mathbb{R}}

\newcommand{\set}[2]{\left\{#1\,\left|\,#2\right.\right\}}

\newcommand{\mmap}[3]{#1:\,#2\rightrightarrows #3\,}
\newcommand{\card}{\ensuremath{\operatorname{card}}}

\DeclareMathOperator{\Id}{Id}
\DeclareMathOperator{\dist}{dist}

\newtheorem{alg}[theorem]{Algorithm}

\begin{document}

\title{Phase retrieval with sparse phase constraint\thanks{The research leading to these results has received funding from the European Research Council under the European Union's Seventh Framework Programme (FP7/2007-2013)/ERC grant agreement No. 339681.
DRL was supported in part by German Israeli Foundation Grant G-1253-304.6 and Deutsche Forschungsgemeinschaft Collaborative Research Center SFB755.
}
}


\author{Nguyen Hieu Thao         \and
        David Russell Luke
        \and \quad \quad Oleg Soloviev
        \and Michel Verhaegen
}


\institute{N. H. Thao \at
              Delft Center for Systems and Control,
 			  Delft University of Technology,
 			  2628CD Delft, The Netherlands.
              Department of Mathematics,
              School of Education, Can Tho University,
              Can Tho, Vietnam. \\
              Tel.: +31-6-84986977\\
              \email{h.t.nguyen-3@tudelft.nl, nhthao@ctu.edu.vn}           
           \and
           D. R. Luke  \at
              Institut f\"ur Numerische und Angewandte Mathematik, Universit\"at G\"ottingen, 37083 G\"ottingen, Germany.
              \email{r.luke@math.uni-goettingen.de}
           \and
           O. Soloviev \at
           Delft Center for Systems and Control,
 			  Delft University of Technology,
 			  2628CD Delft, The Netherlands.
              ITMO University, Kronverksky 49, 197101 St Petersburg, Russia.
              Flexible Optical B.V., Polakweg 10-11, 2288 GG Rijswijk, The Netherlands.
              Oleg Soloviev is one of the two main contributors.
             \email{o.a.soloviev@tudelft.nl}
             \and
           M. Verhaegen \at
           Delft Center for Systems and Control,
 			  Delft University of Technology,
 			  2628CD Delft, The Netherlands.
             \email{m.verhaegen@tudelft.nl}
}

\date{Received: date / Accepted: date}

\maketitle

\begin{abstract}
For the first time, this paper investigates the phase retrieval problem with the assumption that the phase (of the complex signal) is sparse in contrast to the sparsity assumption on the signal itself as considered in the literature of sparse signal processing.
The intended application of this new problem model, which will be conducted in a follow-up paper, is to practical phase retrieval problems where the aberration phase is sparse with respect to the orthogonal basis of Zernike polynomials.
Such a problem is called sparse phase retrieval (SPR) problem in this paper.
When the amplitude modulation at the exit pupil is uniform, a new scheme of sparsity regularization on phase is proposed to capture the sparsity property of the SPR problem.
Based on this regularization scheme, we design and analyze an efficient solution method, named SROP algorithm, for solving SPR given only a single intensity point-spread-function image.
The algorithm is a combination of the Gerchberg-Saxton algorithm with the newly proposed sparsity regularization on the phase.
The latter regularization step is mathematically a rotation but with direction varying in iterations.
Surprisingly, this rotation is shown to be a metric projection on an auxiliary set which is independent of iterations.
As a consequence, SROP algorithm is proved to be the cyclic projections algorithm for solving a feasibility problem involving three auxiliary sets.
Analyzing regularity properties of the latter auxiliary sets, we obtain convergence results for SROP algorithm based on recent convergence theory for the cyclic projections algorithm.
Numerical results show clear effectiveness of the new regularization scheme for solving the SPR problem.
\end{abstract}

\section{Introduction}\label{s:intro}
\emph{Phase retrieval} is an inverse problem of recovering a complex signal from one or several measured intensity patterns.
It appears in many scientific and engineering fields, including astronomy, crystallography,  microscopy, optical manufacturing, and adaptive optics \cite{Fie82,Fie13,LukBurLyo02,SheEldCohChaMiaSeg15}.
An important application of phase retrieval is to quantify the properties of an imaging system via its generalized pupil function (GPF).
The fundamental advantage of this approach compared to those using intensity point spread functions (PSFs) or intensity optical transfer functions is that it is modifiable and automatically includes specific characterizations of the imaging system under investigation.
The generic formulation of the phase retrieval problem is \cite{SheEldCohChaMiaSeg15}:
\begin{equation}\label{generic PR}
\mbox{find}\quad x \in \mathbb{C}^{N}\quad \mbox{such that}\quad
|Mx|^2 = b + \varepsilon,
\end{equation}
where $b$ is the measurement data, $M$ is a known complex matrix of compatible size and $\varepsilon$ is unknown measurement noise.

Depending on the parameterization, phase retrieval problem can be formulated in either the zonal form or the modal form \cite{DoeNguVer18}.
In the zonal form approach, the aberration phase (equivalently, the GPF) is simply expressed via the canonical (pixel-wise) basis.
One of the main advantages of this formulation is that the matrix $M$ possesses desirable properties such as unitary or isometry which is helpful for understanding numerical behavior of solution methods \cite{BauComLuk02,LevSta84,LukBurLyo02,Luk05,LukSabTeb18}.
In the modal form approach, one can either parameterize the phase aberration only \cite{Southwell1977}, or represent the GPF as a linear combination of some basis functions, e.g., the extended Nijboer-Zernike basis functions \cite{Braat:05} or the radial basis functions, see for example  \cite{DoeNguVer18,PisGupSolVer18}.
In the subsequent discussion, by the modal form we restrict ourselves to the parameterization of the aberration phase (not the GPF) in terms of Zernike polynomials.
The advantage of this approach is that the dimension of the parameter variable $x$ does not depend on the size of the problem (the size of the sample grid) but rather the number of the basis functions in use to approximate the aberration phase.
This makes the modal form approach a feasible choice for very large-scale phase retrieval in practice since the aberration phase can often be approximated with a good accuracy using a small number of first Zernike modes.
In the context of \emph{phase retrieval with sparse phase constraint}, both of the zonal and modal form approaches play their own important roles.
In the zonal formulation, the sparse phase constraint is with respect to the canonical basis and hence the analysis of the problem as well as its solution methods can be obtained nicely without much technical argument.
This paper is devoted to this task.
It is worth mentioning that the analysis in this formulation can be useful, for instance, in characterizing phase-only objects such as microlenses, phase-contrast microscopy, optical path difference microscopy and in Fourier ptychography, where the phase object (e.g., blood cells) occupies less than 10\% of the whole filed.
However, detailed considering of such applications is beyond the scope of this paper.
In the modal formulation, the sparse phase constraint is with respect to the orthogonal basis of Zernike polynomials, as explained earlier, this is an often case of practical application and indeed has drawn our first attention to \emph{phase retrieval with sparse phase constraint}.
This application task will be conducted in a follow-up paper.

Since a PSF image captures the complex object GPF via the squared amplitude of its Fourier transform, multiple PSF images are, on one hand, usually needed for retrieving the phase of the GPF up to a total piston term.
On the other hand, simultaneously measuring several PSF images is not always a simple task for many applications in practice, especially for those involving real time imaging systems.
A fundamental reason is that the use of double exposure procedures for such a measurement process would inevitably generate unwanted noise to the obtained data.
It is known that under certain circumstances the phase can be estimated from a single intensity PSF image, see for example \cite{GerSax72,BruVisVer17}.
In this paper, we are interested in phase retrieval with sparse phase constraint given only a single intensity PSF image.
To continue the discussion it is essential to clearly distinguish the phase retrieval with sparse phase constraint from the \emph{sparse signal recovery} (SSR) problem which is widely known and also well investigated in the literature \cite{MukSee12,SheBecEld14,PauBecEldSab18}.
The latter problem (which is also referred to as sparse phase retrieval in the literature, but not in this paper!) is the problem \eqref{generic PR} with the assumption that $x$ is sparse while the former one is also the problem \eqref{generic PR} but with the assumption that the phase (argument) of $x$ is sparse.
Since the latter assumption on the sparsity of $\arg(x)$ does not imply any sparsity property of $x$,
the sparse phase retrieval problem considered in this paper is different from the known SSR problem.
Indeed, to our awareness this paper considers the phase retrieval with sparse phase constraint for the first time.
As a negative consequence, we found no existing phase retrieval algorithms which can be adapted efficiently for this problem.

In summary, this paper deals with phase retrieval problem with sparse phase constraint given a single intensity PSF image in the zonal form formulation as below, and the research goal is to design and analyze an efficient solution algorithm for it.

\textbf{Phase retrieval with sparse phase constraint.} This paper considers the sparse phase retrieval (SPR) problem of
\begin{equation}\label{SPR}
\begin{aligned}
\mbox{finding}\quad &x = \chi\odot \exp\paren{j\paren{\Phi^d+\Phi}}\in \mathbb{C}^{n\times n}
 \\
&\mbox{such that}\quad |\fft(x)|^2 = b + \varepsilon\quad \mbox{and}\quad \norm{\Phi}_0 \le s.
\end{aligned}
\end{equation}
where $\chi\in \mathbb{R}_+^{n\times n}$ is the known amplitude modulation, $\Phi^d\in \mathbb{R}^{n\times n}$ is a known phase diversity, $\Phi\in (-\pi,\pi]^{n\times n}$ is the unknown variable, $\fft$ is the (discrete) \emph{2-dimensional Fourier transform}, $b\in \mathbb{R}_+^{n\times n}$ is the measured intensity image of $x$, $\varepsilon \in \mathbb{R}^{n\times n}$ is unknown noise, and $s$ is a priori upper bound of the sparsity level of $\Phi$.
In this paper, $\chi$ is assumed to be constant.
For the simplicity of terminology, when $x = \chi\odot \exp\paren{j\paren{\Phi^d+\Phi}}$ is a solution to \eqref{SPR}, $\Phi$ is often called a \emph{phase solution} to that problem accordingly.
Note that the phase diversity $\Phi^d$ has an important role in applications though it brings no differences to \eqref{SPR} in terms of mathematics.

Inspired by the success of the sparsity regularization scheme incorporated into projection algorithms for solving the \emph{affine sparse feasibility} problem reported in \cite{HesLukNeu14}, in this paper we show that such a regularization strategy but applied to the phase instead of the signal is efficient for SPR provided that the amplitude modulation $\chi$ is uniform.
It is worth emphasizing that the latter condition is crucial for all the analysis and discussion in this paper.
This new regularization scheme is mathematically described as a rotation within $\mathbb{C}^{n\times n}$ but with direction varying in iterations.
We will prove in Section \ref{s:analysis} that this rotation is indeed a metric projection on an auxiliary set which is independent of iterations.
This result is rather unexpected since the sparsity regularization applied to the phase on $(-\pi,\pi]^{n\times n}$ turns out to be a metric projection on the underlying complex Hilbert space $\mathbb{C}^{n\times n}$.

The remainder of paper is organized as follows. Basic mathematical notation will be introduced in the rest of this section.
Section \ref{s:sparsity constraint} presents the new scheme of sparsity regularization applied to the phase.
In Section \ref{s:SROP} we design a natural and efficient algorithm for solving SPR, named SROP algorithm, based on the Gerchberg-Saxton algorithm and the regularization scheme introduced in Section \ref{s:sparsity constraint}.
In Section \ref{s:analysis}, SROP algorithm is proved to be the cyclic projections algorithm for solving a feasibility problem involving three auxiliary sets.
Based on recent knowledge about the latter algorithm reported in \cite{LukNguTam16}, we obtain convergence results for SROP algorithm by analyzing regularity properties of the three auxiliary sets.
Section \ref{s:numerical simulation} numerically demonstrates the effectiveness of the proposed sparsity regularization scheme by showing better empirical performance of SROP algorithm compared to the corresponding algorithm but without the regularization step.
The latter algorithm is nothing else but the famous Gerchberg-Saxton algorithm \cite{GerSax72}.
Since the SPR problem \eqref{SPR} is first considered in this paper and no existing phase retrieval algorithms are relevant to it, the focus of Section \ref{s:numerical simulation} is to demonstrate the success of our new regularization scheme rather than to compare SROP algorithm with the GS algorithm.
Concluding remarks are presented in Section \ref{s:conclusion}.

\textbf{Mathematical notation.} The underlying space in this paper is a complex Hilbert space $\mathcal{H}=\Cbb^{n\times n}$.
The Frobenius norm is denoted $\|\cdot\|$. The $0$-norm  $\norm{\cdot}_0$ of an object is the number of its nonzero entries.
The element-wise multiplication is denoted $\odot$.
The element-wise division $\frac{\;\cdot\;}{\;\cdot\;}$,
the element-wise absolute value $|\cdot|$ and the element-wise square root $\sqrt{\;\;}$ operations are also frequently used in this paper but without need for extra notation.
The distance to a set $\Omega\subset \mathcal{H}$
is defined by
\begin{equation*}%
\dist(\cdot, \Omega) \colon \mathcal{H}\to\Rbb_+\colon x\mapsto \inf_{w\in\Omega}\norm{x-w}
\end{equation*}%
and the set-valued mapping
\begin{equation*}%
\mmap{P_\Omega}{\mathcal{H}}{\Omega}\colon
x\mapsto \set{w\in \Omega}{\norm{x-w}=\dist(x,\Omega)}
\end{equation*}%
is the corresponding \emph{projection operator}.
A selection $w\in P_\Omega(x)$ is called a {\em projection} of $x$ on $\Omega$.
An iterative sequence $x_{k+1}\in T(x_k)$ generated by an operator/algorithm $T:\mathcal{H}\rightrightarrows \mathcal{H}$ is said to \emph{converge R-linearly} to $x^*$ with rate $c\in [0,1)$ if there is a constant $\gamma>0$ such that
\[
\norm{x_k-x^*} \le \gamma c^k\quad \forall k\in \Nbb.
\]
Dealing with objects of size $n\times n$, we frequently use the set of all indices
\begin{equation}\label{J big}
\mathcal{J} \equiv \{(r,c) \mid 1\le r,c,\le n\}
\end{equation}
and the set of all $s$-element subsets of $\mathcal{J}$ by
\[
\mathcal{J}_s \equiv \{J \subset \mathcal{J} \mid J \text{ has } s \text{ elements}\}.
\]
For a real or complex object $\mathcal{O}$ of size ${n\times n}$, we define the set (see \cite[equation (33)]{BauLukPhaWan14} or \cite[equation (17)]{HesLukNeu14})
\begin{equation}\label{C_s}
  C_s(\mathcal{O}) \equiv \left\{J \in \mathcal{J}_s \mid \min_{\xi\in J}|\mathcal{O}(\xi)| \ge \max_{\xi\in \mathcal{J}\setminus J}|\mathcal{O}(\xi)| \right\}.
\end{equation}

Our other basic notation is standard; cf. \cite{DonRoc14,Mord06,VA}.
The open unit ball in $\mathcal{H}$ is denoted $\mathbb{B}$.
$\mathbb{B}_\delta(x)$ stands for the open ball with radius $\delta>0$ and center $x$.

\section{Sparsity regularization on phase}\label{s:sparsity constraint}

This section analyzes the novel idea of this paper.
In the literature of sparse signal recovery and sparse feasibility, e.g., \cite{PauBecEldSab18,MukSee12,HesLukNeu14}, \emph{sparsity constraint on the signal} has been shown to be useful for the retrieval process.
In this section we reveal that \emph{sparsity constraint on the phase} can be incorporated to efficiently solve SPR problem \eqref{SPR}.
In view of the profound difference between SPR and SSR explained in Section \ref{s:intro}, the analysis developed in this section is new to our awareness.
\bigskip

To begin, let us define on $\mathbb{C}^{n\times n}$ the three constraint sets in accordance with the SPR problem \eqref{SPR} as follows:
\begin{equation}\label{Omega_i}
  \begin{aligned}
  &\Omega_1 \equiv \parennn{x\in \mathbb{C}^{n\times n}\mid |\fft(x\odot\exp(j\Phi^d))|^2 = b},
  \\
  &\Omega_2 \equiv \parennn{x\in \mathbb{C}^{n\times n}\mid |x| = \chi},
  \\
  &\Omega_3 \equiv \parennn{x\in \Omega_2 \mid \|\arg(x)\|_0 \le s},
\end{aligned}
\end{equation}
where $\arg(x)$ denotes the element-wise argument of $x$ taking values in $(-\pi,\pi]$ and $s\in\mathbb{N}$.

The set $\Omega_3$ describes the sparsity constraint on the phase and a natural idea to incorporate it into solution methods for SPR \eqref{SPR} is to perform an additional projection step onto this set.
For uniform amplitude modulation $\chi$ as assumed, the next fundamental lemma provides an explicit form of $P_{\Omega_3}$.

\begin{lemma}[projector $P_{\Omega_3}$]\label{l:P_Omega_3}
For any $x=\chi \odot \exp\paren{j\Phi} \in \Omega_2$, it holds that
\begin{multline}\label{P_chi}
P_{\Omega_3}(x) = \big{\{}\chi \odot \exp\paren{j\varphi} \mid \exists J \in C_s(\Phi)\\
 \text{ such that } \varphi = \Phi \text{ on } J \text{ and vanishes elsewhere}\big{\}},
\end{multline}
where the mapping $C_s(\cdot)$ is defined at \eqref{C_s}.
\end{lemma}
\begin{proof}
Let us take any $J \in C_s(\Phi)$ and $\varphi \in \Rbb^{n\times n}$ such that
\[
\varphi(\xi) =
\begin{cases}
\Phi(\xi) & \text{ if } \xi\in J,\\
0 & \text{ if } \xi \in \mathcal{J}\setminus J,
\end{cases}
\]
where the set of all indices $\mathcal{J}$ is defined at \eqref{J big}.
It suffices to check that
\begin{equation*}
  \norm{x-\chi\odot \exp\paren{j\varphi}} \le \norm{x-w}\quad \forall w\in \Omega_3.
\end{equation*}
Since $\chi$ is uniform as assumed throughout this paper, the above condition amounts to
\begin{equation}\label{P_Omega_3}
  \norm{\exp\paren{j\Phi}-\exp\paren{j\varphi}} \le \norm{\exp\paren{j\Phi}-\exp\paren{j\phi}}\quad \forall \phi\in \Rbb^{n\times n} \text{ with } \norm{\phi}_0\le s.
\end{equation}
We first claim that
\begin{equation}\label{xi<xi^c}
|\exp(j\Phi(\xi^c))-1|\le |\exp(j\Phi(\xi))-1|
\quad \forall \xi\in J,\;\forall \xi^c\in \mathcal{J}\setminus J.
\end{equation}
Indeed, thanks to $J\in C_s(\Phi)$ and the definition \eqref{C_s} of $C_s(\Phi)$ it holds that
\begin{equation}\label{Psi_k}
  0\le|\Phi(\xi^c)|\le |\Phi(\xi)|\le \pi \quad \forall \xi\in J,\;\forall \xi^c\in \mathcal{J}\setminus J.
\end{equation}
Then we have for all $\xi\in J$ and $\xi^c\in \mathcal{J}\setminus J$ that
\begin{align}\nonumber
|\exp(j\Phi(\xi^c))-1| &= \sqrt{\paren{\cos\Phi(\xi^c)-1}^2 + \sin^2\Phi(\xi^c)}
\\\label{sqrt(2-cos)}
&= \sqrt{2-2\cos|\Phi(\xi^c)|}
\\\nonumber
&\le \sqrt{2-2\cos|\Phi(\xi)|}
\\\nonumber
&= |\exp(j\Phi(\xi))-1|,
\end{align}
where the only inequality follows from \eqref{Psi_k}.
Hence \eqref{xi<xi^c} has been proved.

We now can prove \eqref{P_Omega_3}.
Note from the definition of $\varphi$ that
\begin{align}\label{varphi=}
    \norm{\exp\paren{j\Phi}-\exp\paren{j\varphi}} = \sum_{\xi^c\in \mathcal{J}\setminus J}|\exp(j\Phi(\xi^c))-1|.
\end{align}
It is also clear that
\begin{align}\label{phi>}
    \norm{\exp\paren{j\Phi}-\exp\paren{j\phi}} \ge \sum_{\phi(\xi)=0}|\exp(j\Phi(\xi))-1|.
\end{align}
Since $\norm{\phi}_0\le s$, it holds that
\[
 \card\paren{\{\xi \mid \phi(\xi)=0\}}\ge n^2-s = \card\paren{\{\xi^c \mid \xi^c\in \mathcal{J}\setminus J\}},
 \]
where $\card(\cdot)$ denotes the cardinal of set.
The inequality \eqref{xi<xi^c} then implies that
\begin{align}\label{sum<sum}
    \sum_{\xi^c\in \mathcal{J}\setminus J}|\exp(j\Phi(\xi^c))-1| \le \sum_{\phi(\xi)=0}|\exp(j\Phi(\xi))-1|.
\end{align}
The desired estimate \eqref{P_Omega_3} now follows from
\eqref{varphi=}, \eqref{phi>} and \eqref{sum<sum}, and the proof is complete.
\hfill{$\square$}
\end{proof}

\begin{remark}\label{r:only for uniform}
Lemma \ref{l:P_Omega_3} is crucial for the analysis in this paper and rather unexpected since the  sparsity regularization applied to the phase object on $(-\pi,\pi]^{n\times n}$ turns out to be a metric projection on the underlying space $\mathbb{C}^{n\times n}$.
It is worth emphasizing that the calculation of $P_{\Omega_3}$ is technically nontrivial for general (nonuniform) amplitude modulation $\chi$.
\end{remark}

\section{Proposed algorithm for SPR}\label{s:SROP}

In accordance with the new analysis of Section \ref{s:sparsity constraint}, this section proposes an efficient solution algorithm for solving the sparse phase retrieval \eqref{SPR}.

\begin{alg}\label{al:SROP}[Sparsity regularization on phase (SROP) algorithm for SPR problem \eqref{SPR}]\\
\emph{Input:}

$b \in \mathbb{R}_+^{n\times n}$ --- intensity PSF image

$s$ --- sparsity parameter

$\Phi_0$ --- initial guess for $\Phi$

$\tau$ --- tolerance threshold.
\\
\emph{Iteration process}: given $\Phi_k$
\begin{enumerate}[(i)]
  \item\label{al:x_k} $x_k = \chi\odot\exp\paren{j\Phi_k}$
  \item\label{al:fft} $X_k = \fft\paren{x_k\odot \exp\paren{j\Phi^d}}$ --- phase diversity applied and Fourier transform
  \item\label{al:scale} $Y_k = \sqrt{b}\odot \frac{X_k}{|X_k|}$ --- amplitude constraint
  \item\label{al:ifft} $y_k = \exp\paren{-j\Phi^d}\odot\fft^{-1}(Y_k)$ --- inverse Fourier transform and phase diversity corrected
  \item\label{al:amplitude} $z_k = \chi\odot\exp\paren{j\arg\paren{y_k}}$ --- amplitude modulation
  \item\label{al:index} find $J \in C_s(\varphi)$ where $\varphi \equiv \arg(z_k)$, and set
   \begin{equation}\label{truncate}
     \Phi_{k+1}(\xi) =
     \begin{cases}
     \varphi(\xi) & \text{ if } \xi\in J,\\
     0 & \text{ if } \xi\in \mathcal{J}\setminus J.
     \end{cases}
   \end{equation}
\end{enumerate}
\emph{Stopping criterion:} $\norm{\Phi_k-\Phi_{k+1}}<\tau$ or $k$ exceeds the maximal number of iterations.
\\
\emph{Output:} $\widehat{\Phi} = \Phi_{end}$ --- the estimate phase.
\end{alg}

Algorithm \ref{al:SROP} is nothing else but a direct combination of the Gerchberg-Saxton (GS) algorithm (steps \ref{al:x_k}--\ref{al:amplitude}) and the sparsity regularization on the phase (step \ref{al:index}) analyzed in Section \ref{s:sparsity constraint}.
When applied to SPR \eqref{SPR}, the latter regularization step brings significant features of SROP compared to the original GS method, see Section \ref{s:numerical simulation}.

The convergence analysis of the SROP algorithm is provided in the next section.

\section{Convergence analysis}\label{s:analysis}

In this section we will establish local convergence results for SROP algorithm by first showing that it is indeed the cyclic projections algorithm involving three auxiliary sets.
Then the convergence theory for the cyclic projections method recently developed in \cite{LukNguTam16} can be applied to SROP algorithm.

Let us first shows that SROP algorithm is nothing else but the cyclic projections method $P_{\Omega_3}P_{\Omega_2}P_{\Omega_1}$ where the sets $\Omega_i$ $(i=1,2,3)$ are defined at \eqref{Omega_i}.

\begin{lemma}\label{l:SROP=CP} The following statements hold true.
\begin{enumerate}[(i)]
  \item The combination of steps \ref{al:fft}, \ref{al:scale} and \ref{al:ifft} of Algorithm \ref{al:SROP} amounts to $y_k\in P_{\Omega_1}(x_k)$.
  \item Step \ref{al:amplitude} of Algorithm \ref{al:SROP} amounts to $z_{k}\in P_{\Omega_2}(y_k)$.
  \item The combination of steps \ref{al:index} and \ref{al:x_k} of Algorithm \ref{al:SROP} amounts to $x_{k+1}\in P_{\Omega_3}(z_k)$.
\end{enumerate}
In other words, the SROP algorithm is characterized by the fixed point operator $T$ given by:
\begin{equation}\label{T_SROP}
  x \mapsto T(x) \equiv P_{\Omega_3}P_{\Omega_2}P_{\Omega_1}(x)\quad \forall x\in\Cbb^{n\times n}.
\end{equation}
\end{lemma}
\begin{proof}
$(i)$ We first observe that
\begin{equation}\label{Omega1'}
  \Omega_1 = \{x\in \mathbb{C}^{n\times n} \mid \fft\paren{x\odot \exp\paren{j\Phi^d}} \in C\},
\end{equation}
where the constraint set $C$ is given by
\begin{equation}\label{C}
  C \equiv \parennn{y\in \mathbb{C}^{n\times n}\mid |y|^2 = b}.
\end{equation}
Let us denote the linear operator $D:\mathbb{C}^{n\times n}\to \mathbb{C}^{n\times n}$ by
\begin{equation}\label{D}
  x \mapsto D(x) \equiv x\odot \exp\paren{j\Phi^d}\quad \forall x\in \mathbb{C}^{n\times n}.
\end{equation}
It is clear that $D$ is a unitary transform since the energy of any signal in $\mathbb{C}^{n\times n}$ is invariant with respect to $D$.
Its inverse is also unitary and takes the form $D^{-1}(x)=x\odot \exp\paren{-j\Phi^d}$.
Now using \eqref{Omega1'}, we obtain that
\begin{align}\label{P_Omega_1=}
  P_{\Omega_1}(x) = D^{-1}\circ\fft^{-1}\paren{P_C\paren{\fft\circ D(x)}} \quad \forall x\in \mathbb{C}^{n\times n}.
\end{align}
Comparing to steps \ref{al:fft}, \ref{al:scale} and \ref{al:ifft} of Algorithm \ref{al:SROP} and using \eqref{P_Omega_1=}, we have that
\begin{align*}
 y_k &= \exp\paren{-j\Phi^d}\odot\fft^{-1}\paren{
\sqrt{b}\odot \frac{X_k}{|X_k|}}
   \\
     &= \exp\paren{-j\Phi^d}\odot\fft^{-1}\paren{
\sqrt{b}\odot \frac{\fft\paren{x_k\odot \exp\paren{j\Phi^d}}}{|\fft(x_k\odot \exp\paren{j\Phi^d})|}}
   \\
     &= D^{-1}\circ \fft^{-1}\paren{P_C\paren{\fft\circ D(x_k)}} = P_{\Omega_1}(x_k) \quad \forall k\in \Nbb.
\end{align*}

$(ii)$ This statement follows from step \ref{al:amplitude} of Algorithm \ref{al:SROP} and the definition of $\Omega_2$.

$(iii)$ This statement follows from Lemma \ref{l:P_Omega_3}.
\hfill{$\square$}
\end{proof}

To this end, SROP algorithm is the cyclic projections algorithm for solving the feasibility problem of
\begin{equation}\label{auxi_FP}
  \mbox{finding}\quad x \in \Omega_1 \cap \Omega_2 \cap \Omega_3,
\end{equation}
which is an equivalent reformulation of the SPR problem \eqref{SPR} (in the noise-free case) since $x$ is a solution to \eqref{SPR} if and only if it is a solution to \eqref{auxi_FP}.

As a result, one can fully applies the convergence theory for the cyclic projections algorithm developed in \cite[Section 3.1]{LukNguTam16} to SROP algorithm.
It is worth mentioning that the above mentioned theory also encompasses the inconsistent feasibility which in the setting of this paper corresponds to SROP algorithm applied to SPR \eqref{SPR} with noise.
To avoid quite formidable technical details inherently arising for inconsistent feasibilities, we will present the convergence results in the noise-free setting.
We next present an amount of mathematical material which is sufficient for formulating a local linear convergence criterion in a concise and memorable statement.
Since phase retrieval is a typical nonconvex problem and our approach requires no convex relaxations, we can only establish local convergence result for SROP algorithm as iterations generated by the fixed point operator $T$ given by \eqref{T_SROP}.

Recall that a set $\Omega$ is called \emph{prox-regular} at a point $x^*\in \Omega$ if the projector $P_{\Omega}$ is single-valued around $x^*$ \cite{PolRocThi00}.
The analysis regarding prox-regularity in the context of the phase retrieval problem was first given in \cite{Luk08}.

The following lemma provides the geometry of the auxiliary feasibility problem \eqref{auxi_FP}.

\begin{lemma}\label{l:prox-reg} The following statements regarding the sets $\Omega_i$ ($i=1,2,3$) hold true.
\begin{enumerate}
  \item[(i)] The set $\Omega_1$ is prox-regular at every point of it.
  \item[(ii)] The set $\Omega_2$ is prox-regular at every point of it.
  \item[(iii)] If $s$ is the sparsity of the solutions with the sparsest phase to problem \eqref{auxi_FP}, then the set $\Omega_3$ is prox-regular at every point $x^* \in \Omega_1 \cap \Omega_2 \cap \Omega_3$.
\end{enumerate}
As a consequence, in the setting of item $(iii)$, the sets $\Omega_i$ ($i=1,2,3$) are prox-regular at every solution to problem \eqref{auxi_FP}.
\end{lemma}
\begin{proof}
$(i)$ The proof follows the idea of \cite[Section 3.1]{Luk08}. We first note that the set $C$ defined at \eqref{C} is prox-regular at every point of it since it is the Cartesian product of a number of circles which are typical examples of prox-regularity.
      It is clear from the definition of $\Omega_1$ that $\Omega_1=\fft^{-1}D^{-1}(C)$ with $D$ given by \eqref{D}.
      But $\fft^{-1}D^{-1}$ which is a unitary transform does not affect the geometry properties of $C$ and hence the statement is proved.

$(ii)$ The argument for item $(i)$ also encompasses statement $(ii)$.

$(iii)$ We first note that $\Omega_3$ is the union of $\binom{n^2}{s}$ sets each of which, we call a \emph{component} of $\Omega_3$, is the Cartesian product of $s$ copies of the circle $\{c\in \Cbb \mid |c|=r\}$ and $n^2-s$ singletons $\{1\}$, where $r>0$ is the constant value of $\chi$.

Note that each of the components of $\Omega_3$ is a prox-regular set thanks to statement $(i)$.
Let $x^* \in \Omega_1 \cap \Omega_2 \cap \Omega_3$.
By the definition of the sets $\Omega_i$ ($i=1,2,3$), we have $x^* = \chi\odot\exp\paren{j\Phi^*}$ and $\|\Phi^*\|_0=s$.
Let us denote $J(x^*)$ the set of indices corresponding to nonzero entries of $\Phi^*$.
Since $s$ is the sparsity of the solutions with the sparsest phase to problem \eqref{auxi_FP} as assumed,
there is the unique component of $\Omega_3$ which contains $x^*$, denoted $L({x^*})$.
Since $L({x^*})$ is prox-regular as mentioned above, all we need is to show the existence of a neighborhood $U$ of $x^*$ such that
\begin{equation}\label{Omega3=L}
  \Omega_3 \cap U = L({x^*}) \cap U.
\end{equation}
Let us define
\begin{align}\label{theta}
  \theta &\equiv \min\left\{|\Phi^*(\xi)| : \xi\in J({x^*})\right\} \in (0,\pi],
  \\
  \label{delta}
  \delta &\equiv r\sqrt{2-2\cos\theta\,}\; >\; 0.
\end{align}
We will show that the neighborhood
\begin{align}\label{U}
   U \equiv \parennn{x\in \mathbb{C}^{n\times n} \mid \|x-x^*\|_1<\delta}
\end{align}
of $x^*$ satisfies the desired property \eqref{Omega3=L}.
Since the inclusion $\paren{\Omega_3 \cap U} \supset \paren{L({x^*}) \cap U}$ is trivial, we now prove only the inverse one,
$\paren{\Omega_3 \cap U} \subset \paren{L({x^*}) \cap U}$.
Indeed, take any $x \in \Omega_3 \cap U$.
Note that  $|x| = \chi$ as $x \in \Omega_3 \subset \Omega_2$.
Let us write $x = \chi\odot\exp\paren{j\Phi}$.
Then for every index $\xi\in J({x^*})$, we have by the triangle inequality, \eqref{sqrt(2-cos)}, \eqref{U}, \eqref{theta} and \eqref{delta} successively that
\begin{align*}
  |x(\xi)-\chi(\xi)| &\ge |x^*(\xi)-\chi(\xi)| - |x^*(\xi)-x(\xi)|
  \\
  &= r\sqrt{2-2\cos\Phi^*(\xi)\,} - |x^*(\xi)-x(\xi)|
  \\
  &> r\sqrt{2-2\cos\Phi^*(\xi)\,} - \delta
  \\
  &\ge r\sqrt{2-2\cos\theta\,} - \delta = \delta - \delta = 0.
\end{align*}
This in particular implies that $\Phi(\xi) \neq 0$ for all indices $\xi\in J({x^*})$ and $\|\Phi\|_0=s$.
That is $x\in L({x^*})$ and hence the proof is complete.
\hfill{$\square$}
\end{proof}

It is worth mentioning that the condition on the uniform distribution of the amplitude $\chi$ is essential to both Lemmas \ref{l:SROP=CP} and \ref{l:prox-reg}.

We are now ready to state and prove a local linear convergence criterion for SROP algorithm.

\begin{theorem}[R-linear convergence of Algorithm \ref{al:SROP}]\label{t:convergence}
Let $\Phi^*$ be a sparsest phase solution to SPR problem \eqref{SPR} and $\|\Phi^*\|_0=s$.
Suppose that the set-valued mapping $\Psi \equiv \Id-T$ (where $T$ is given by \eqref{T_SROP}) is \emph{metrically subregular} at $x^* \equiv \chi\odot\exp\paren{j\Phi^*}$ for $0$, i.e., there are constants $\kappa>0$ and $\delta>0$ such that the inequality
\[
\kappa\dist(x,\Psi^{-1}(0)) \le \dist(0,\Psi(x))\quad \forall x\in \Bbb_{\delta}(x^*).
\]
Then every iterative sequence generated by $T$ converges R-linearly to a solution to \eqref{auxi_FP} (equivalently, the SPR problem \eqref{SPR}) provided that the initial point is sufficiently close to $x^*$.
\end{theorem}
\begin{proof}
Since $s$ is the sparsity of the sparsest phase solutions to the SPR problem \eqref{SPR},
$\Phi^*$ is a sparsest phase solution to \eqref{SPR} if and only if $x^*=\chi\odot\exp\paren{j\Phi^*}$ is a solution to \eqref{auxi_FP}.
By Lemma \ref{l:prox-reg} the sets $\Omega_i$ ($i=1,2,3$) are prox-regular at $x^*$.
By \cite[Theorem 2.14 $(ii)$]{HesLuk13}, the projectors $P_{\Omega_i}$ ($i=1,2,3$) are \emph{almost firmly nonexpansive} \cite[Definition 2.2]{LukNguTam16} with violation that can be made arbitrarily small by shrinking the effective neighborhood of $x^*$ if necessary.
Then by \cite[Proposition 2.4 $(iii)$]{LukNguTam16}, there is a neighborhood $U$ of $x^*$ such that
the operator $T$ is almost averaged on $U$ with averaging constant $3/4$ and violation $\varepsilon$ that can be made arbitrarily small. Hence, taking also the metric subregularity of $\Id-T$ at $x^*$ for $0$ into account, we can directly apply \cite[Corollary 2.3]{LukNguTam16} to obtain local linear convergence of $T$. The proof is complete.
\hfill{$\square$}
\end{proof}

The notion of metric subregularity appearing in Theorem \ref{t:convergence} is amongst cornerstones of variational analysis and optimization theory with many important applications \cite{DonRoc14,KlaKum02}. In the context of feasibility problem, the metric subregularity is closely related to the mutual arrangement of the sets at the reference point \cite{HesLuk13,Kru06,Kru09,KruLukNgu18,KruLukNgu17,KruTha15,LukNguTam16}.
A deeper look at this type of regularity for the collection of sets $\{\Omega_1, \Omega_2, \Omega_3\}$ at $x^*$ involved in Theorem \ref{t:convergence} is obviously of importance but beyond the scope of this paper.

\section{Numerical simulation}\label{s:numerical simulation}

This section demonstrates that the sparsity regularization scheme introduced in Section \ref{s:sparsity constraint} is efficient for the sparse phase retrieval problem \eqref{SPR} which is also first considered and analyzed in this paper.
Since none of the existing algorithms in the literature of sparse signal processing can be applied to solve SPR \eqref{SPR} as explained in Section \ref{s:intro}, only comparison between SROP algorithm versus the corresponding one without the step of sparsity regularization on phase can be made.
The latter algorithm is nothing else but the famous Gerchberg-Saxton (GS) algorithm \cite{GerSax72}.
\bigskip

The common parameters for all experiments in this section are below:
\begin{itemize}[$\circ$]
  \item image size: $128 \times 128$ pixels,
  \item circular aperture: $ap$ with diameter $64$ pixels,
  \item amplitude modulation: $\chi$ is uniform with unity value,
  \item phase diversity: $\Phi^d = 4Z_2^0$, where $Z_2^0(\rho)=2\rho^2-1$ is the Zernike polynomial of order $2$ and azimuthal frequency zero,
  \item PSF image: $b = \left|\fft\paren{ap\odot \exp\paren{j\paren{\Phi+\Phi^d}}}\right|^2$, where $\Phi$ is the simulation phase,
  \item initial phase guess: zero phase everywhere,
  \item noise: Poisson noise applied the PSF image $b$ by using the MATLAB \emph{imnoise} function for experiments with noise,
  \item number of iterations: $1200$.
\end{itemize}

The other input data and parameters such as the simulation phase $\Phi$, its sparsity level $\|\Phi\|_0$ and the sparsity guess $s$ will be specified at each of the experiments.

\subsection{Effectiveness of sparsity regularization on phase}\label{subs:effective of SROP}

This section proves the main objective of the paper.

\begin{figure}[H]
	\centering
    \includegraphics[scale=.43]{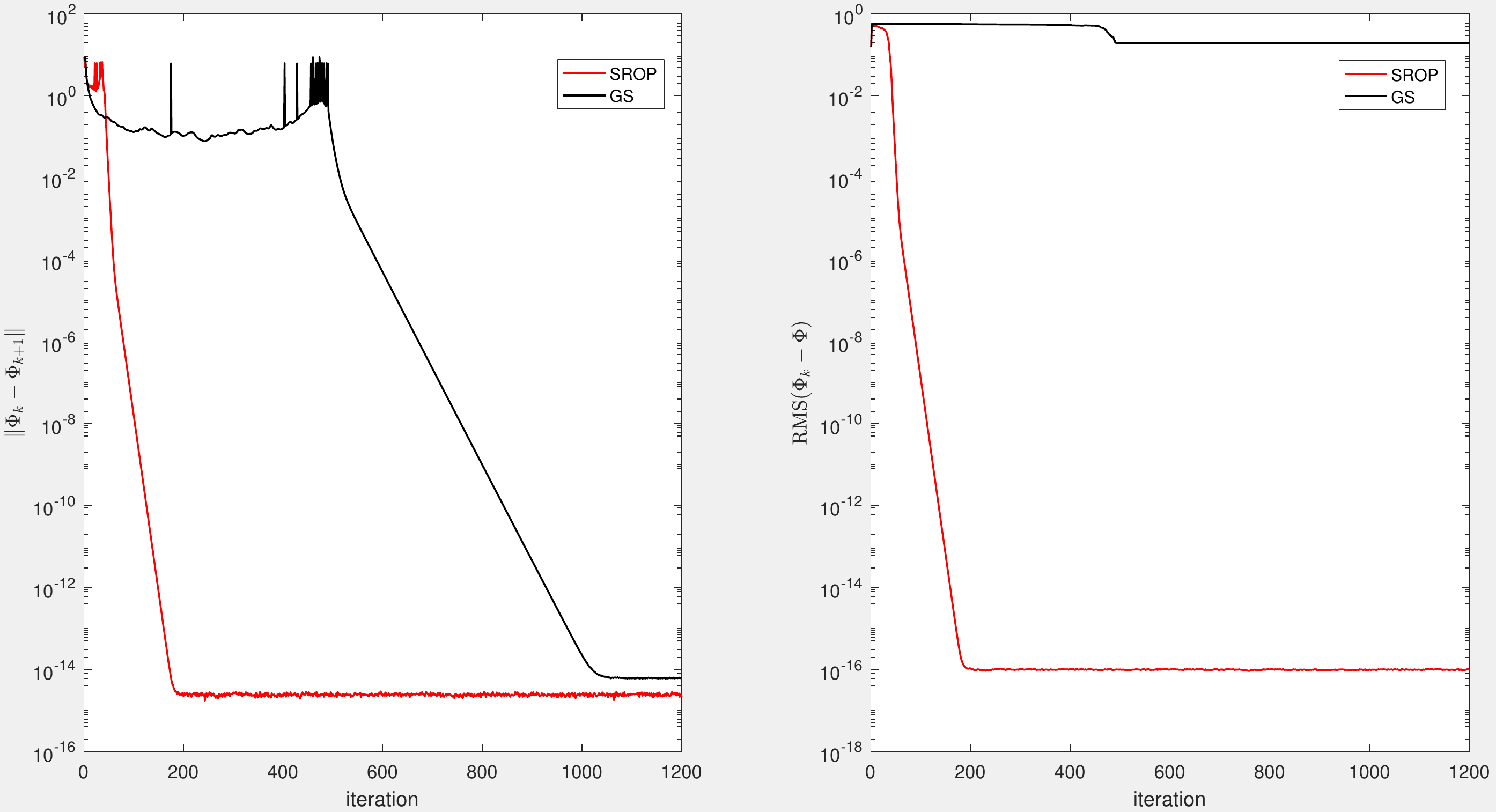}
    \caption{Performance of SROP and GS algorithms for SPR \eqref{SPR} \textbf{without noise}: the left-hand-side figure shows faster convergence and the right-hand-side one shows more accurate restoration of the SROP algorithm.}
    \label{fig:SROP_GS_noiseless_ph}
\end{figure}

We compare SROP algorithm with GS algorithm via two important output figures: 1) the \emph{change} of the distance between two consecutive iterations $\norm{\Phi_k-\Phi_{k+1}}$ which is of interest for understanding convergence properties of the algorithms; 2) the \emph{feasibility gap} RMS$\paren{\Phi_k-\Phi}$ shows the quality of retrieval.
\bigskip

Experiment setup:
\begin{itemize}[$\circ$]
  \item simulation phase: $\Phi$ randomly generated with values in $[-\pi,\pi]$,
  \item sparsity level: $\|\Phi\|_0 = 319$ (the pixel totality is $3168$),
  \item sparsity parameter: $s=335$ (about 105\% of the true sparsity level).
\end{itemize}
\bigskip

The performance of SROP algorithm for SPR \eqref{SPR} compared to GS algorithm in the settings without noise and with Poisson noise is summarized in Figures \ref{fig:SROP_GS_noiseless_ph} and \ref{fig:SROP_GS_noise_ph}, respectively.
It is worth mentioning that the performance shown in Figures \ref{fig:SROP_GS_noiseless_ph} and \ref{fig:SROP_GS_noise_ph} is consistent for all the experiments we have made for random realizations of $\Phi$ described above.
The numerical result clearly demonstrates that our suggestion to apply the sparsity regularization on phase for SPR \eqref{SPR} is effective.

\begin{figure}[H]
	\centering
    \includegraphics[scale=.43]{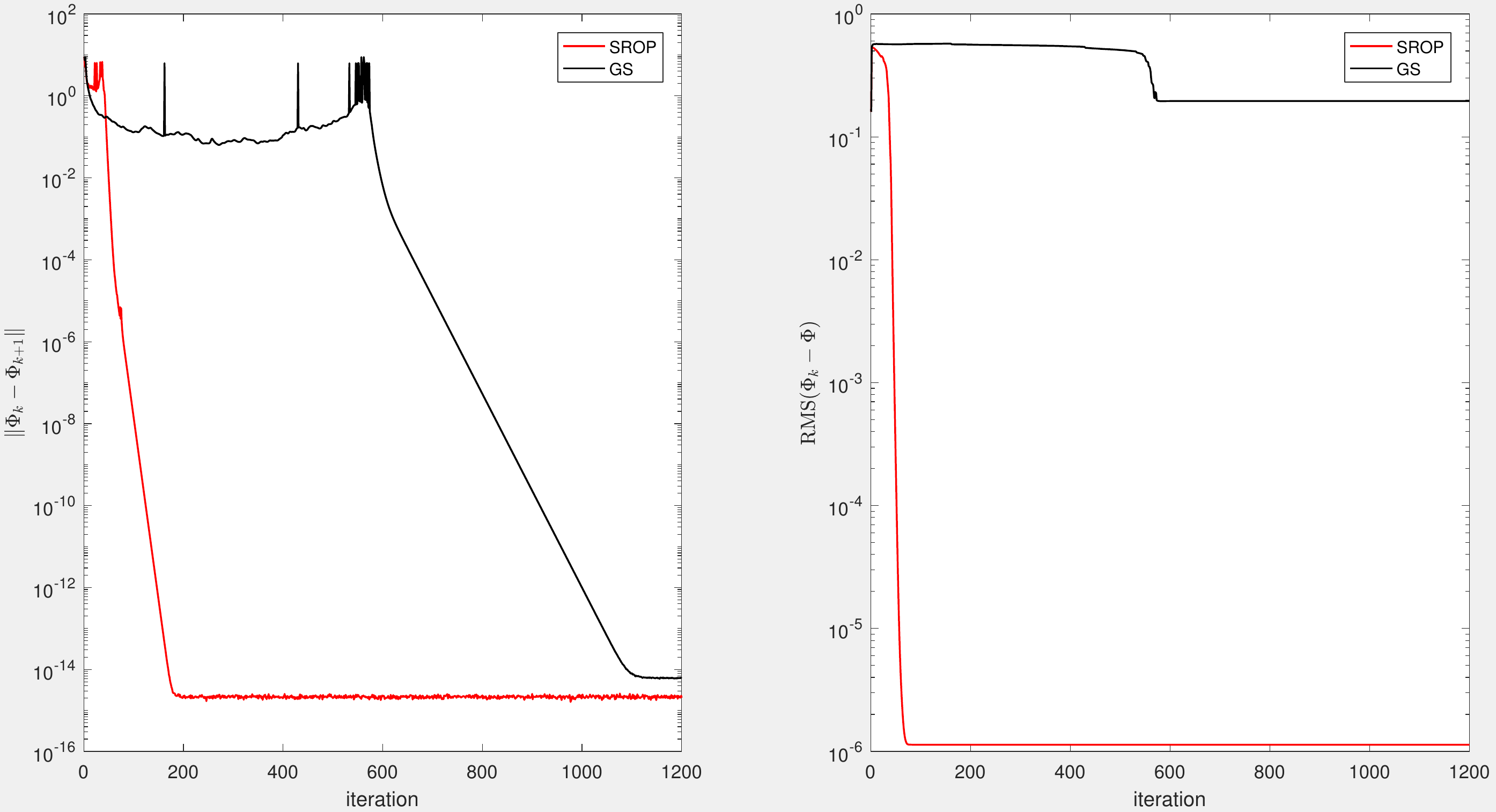}
    \caption{Performance of SROP and SF algorithms for SPR \eqref{SPR} \textbf{with Poisson noise}: the left-hand-side figure shows faster convergence and the right-hand-side one shows better restoration of the SROP algorithm.}
    \label{fig:SROP_GS_noise_ph}
\end{figure}

\subsection{Solvability with respect to sparsity level}\label{subs:solv wrt spar level}

This section numerically addresses the question: to which sparsity level of $\Phi$, the sparsity regularization on phase can be applied efficiently for solving SPR \eqref{SPR}.
Recall the discussion in Section \ref{s:intro} that the analysis of this paper can be applied to characterize phase-only objects in a number of circumstances of optical science where the phase object occupies less than 10\% of the whole filed.
The results of this section, in particular, show that SROP algorithm is efficient up to that level of sparsity of the phase object.

\bigskip

Experiment setup:
\begin{itemize}[$\circ$]
  \item simulation phase: $120$ values of $\Phi$ with different sparsity levels randomly generated with values in $[-\pi,\pi]$,
  \item sparsity level: $120$ different values of $\|\Phi\|_0 = 4k$, for $k=\overline{1,120}$ (consistent with the simulation phase above),
  \item sparsity parameter: $s \simeq 105\%\|\Phi\|_0$,
  \item noise: Poisson noise.
\end{itemize}

\begin{figure}[H]
	\centering
    \includegraphics[scale=.39]{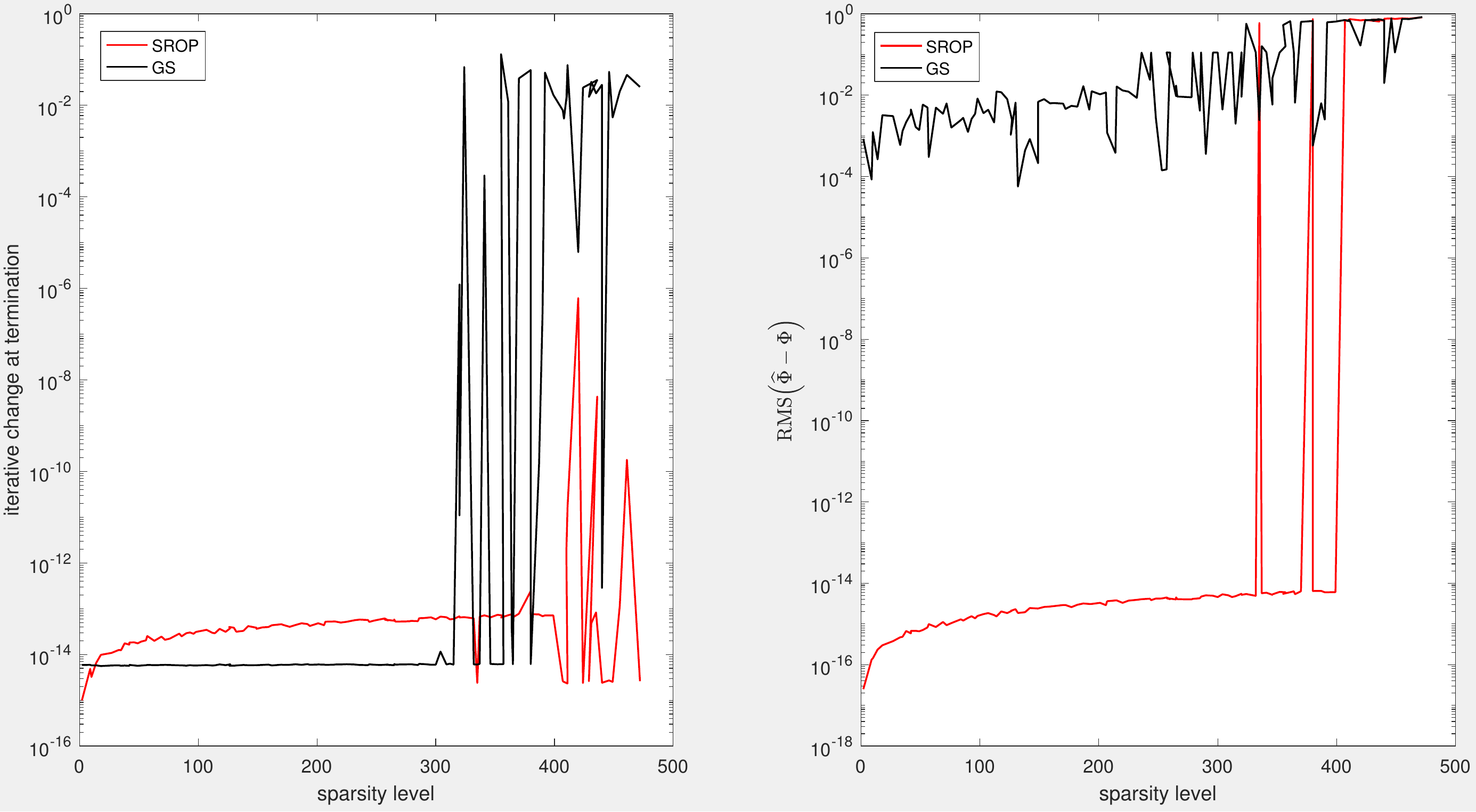}
    \caption{Performance of SROP and GS algorithms for SPR \eqref{SPR} \textbf{without noise} for different levels of sparsity of $\Phi$: the right-hand-side figure shows more accurate restoration of the first algorithm for sparsity level up to $330$ pixels ($>10\%$).}
    \label{fig:SROP_GS_solv_vs_spar_noiseless_ph}
\end{figure}
\begin{figure}[H]
	\centering
    \includegraphics[scale=.39]{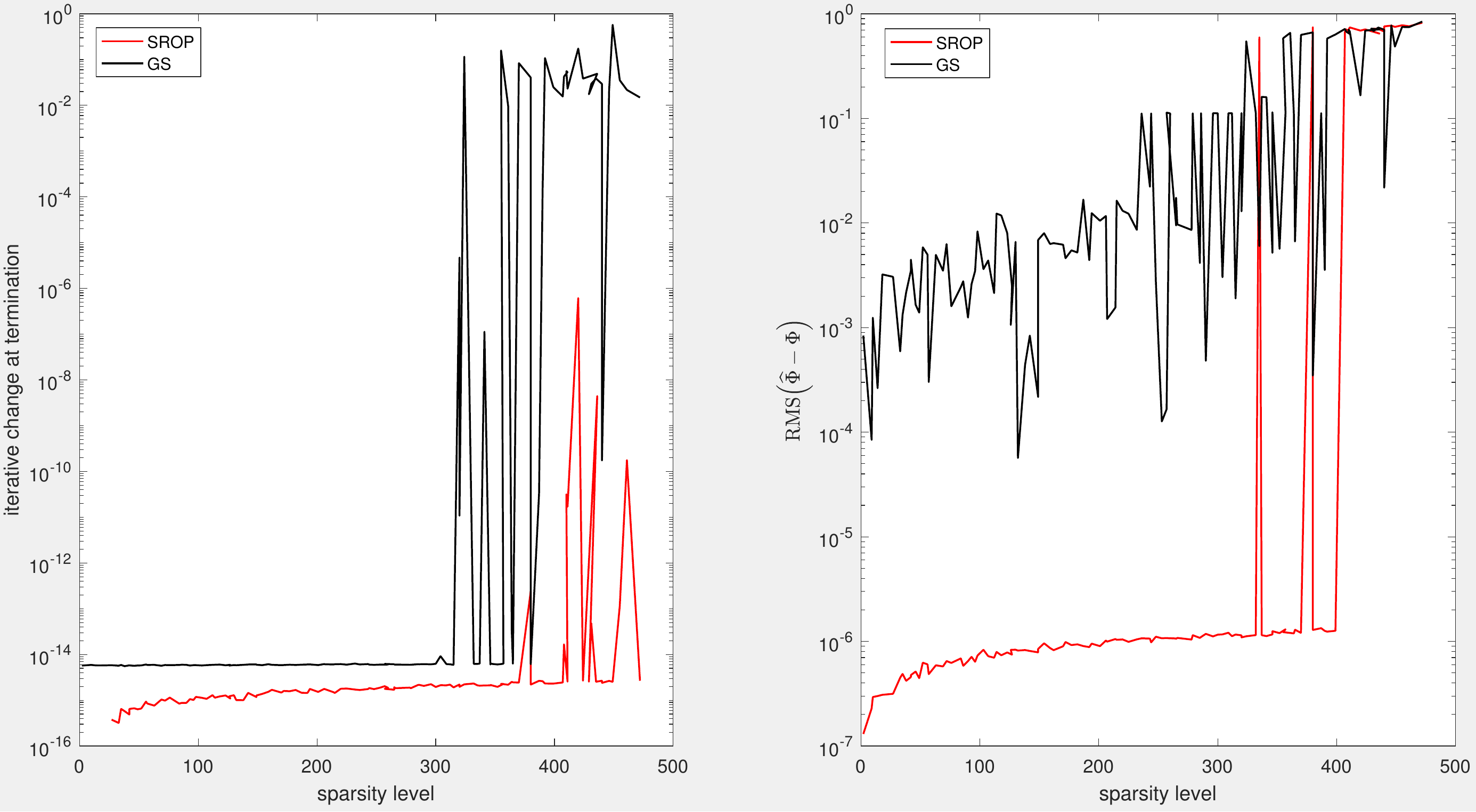}
    \caption{Performance of SROP and GS algorithms for SPR \eqref{SPR} \textbf{with Poisson noise} for different levels of sparsity of $\Phi$: the left-hand-side figure shows faster convergence and the right-hand-side one shows more accurate restoration of the first algorithm for sparsity level up to $330$ pixels ($>10\%$).}
    \label{fig:SROP_GS_solv_vs_spar_noise_ph}
\end{figure}

We report the \emph{change} at termination of the algorithms and the corresponding quality of phase retrieval in terms of  RMS$\paren{\widehat{\Phi}-\Phi}$ for each of the $120$ values of $\Phi$ prescribed above.
The performance of SROP algorithm for SPR \eqref{SPR} compared to GS algorithm in the settings without noise and with Poisson noise is summarized in Figures \ref{fig:SROP_GS_solv_vs_spar_noiseless_ph} and \ref{fig:SROP_GS_solv_vs_spar_noise_ph}, respectively.
The figures on the right-hand-side of Figures \ref{fig:SROP_GS_solv_vs_spar_noiseless_ph} and \ref{fig:SROP_GS_solv_vs_spar_noise_ph} show that the sparsity regularization on phase is very efficient for solving SPR \eqref{SPR} with the sparsity level up to more than $10\%$ ($330$ over the totality of $3168$ pixels).
The figures on the left-hand-side of Figures \ref{fig:SROP_GS_solv_vs_spar_noiseless_ph} and \ref{fig:SROP_GS_solv_vs_spar_noise_ph} just guarantee that the convergence properties of SROP algorithm is not worse than GS algorithm (see also Section \ref{subs:effective of SROP}).

\subsection{Stability with respect to sparsity parameter}\label{subs:stab wrt s_hat}

The sparsity level of $\Phi$ is not known in advance.
Fortunately, this section demonstrates the stability of the SROP algorithm with respect to the sparsity parameter $s$.
The observation is similar to the stability of projection methods for solving \emph{sparse affine feasibility} with respect to the sparsity parameter \cite{HesLukNeu14}.
\bigskip

Experiment setup:
\begin{itemize}[$\circ$]
  \item simulation phase: $\Phi$ randomly generated with values in $[-\pi,\pi]$,
  \item sparsity level: $\|\Phi\|_0 = 319$ (the pixel totality is $3168$),
  \item sparsity parameter: $120$ different values $s=168 + 25k$, for $k=\overline{1,120}$.
\end{itemize}
\bigskip

We report the \emph{change} at termination of SROP algorithm and the corresponding quality of phase retrieval in terms of  RMS$\paren{\widehat{\Phi}-\Phi}$ for each of the $120$ values of $s$ prescribed above.

\begin{figure}[H]
	\centering
    \includegraphics[scale=.43]{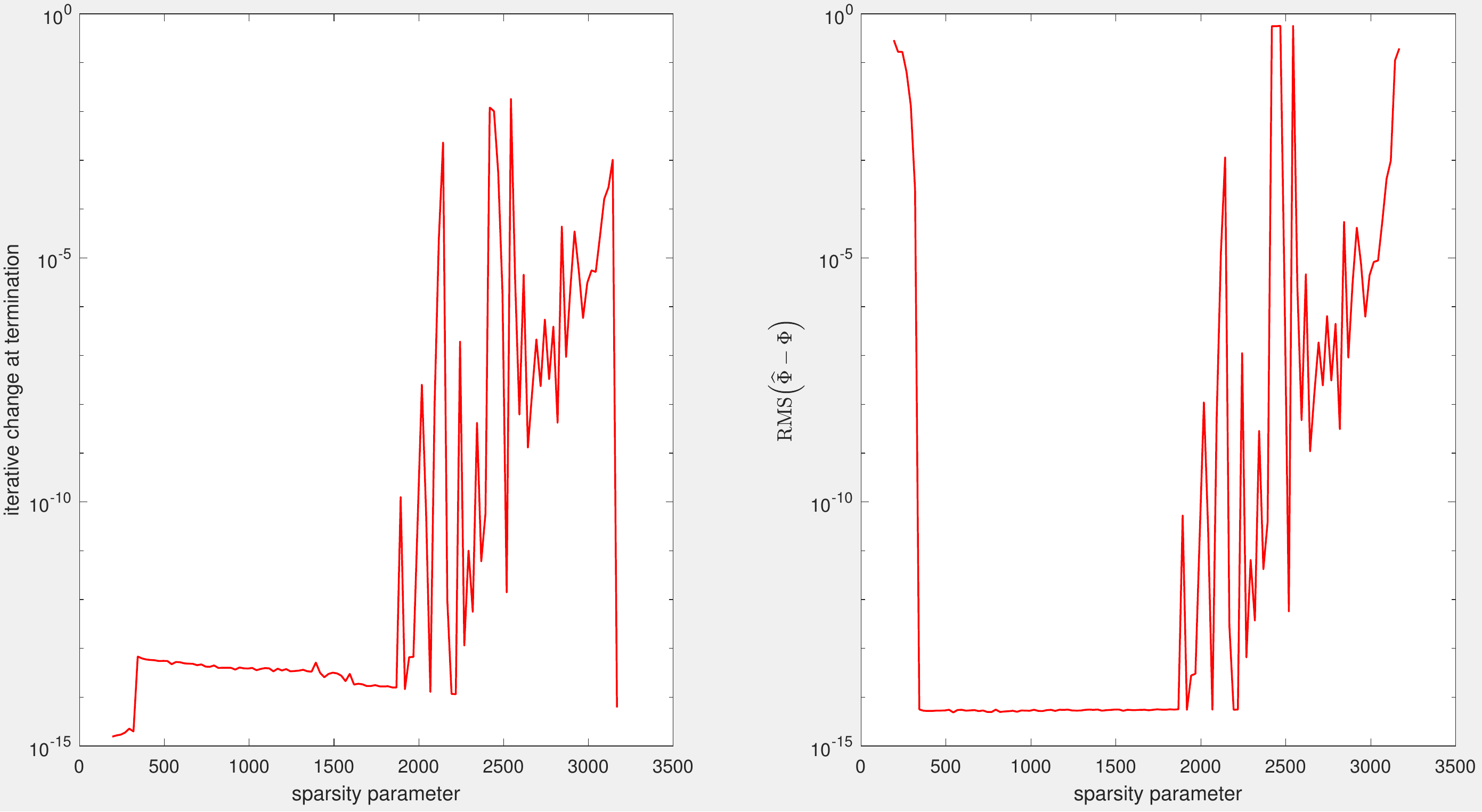}
    \caption{Performance of SROP algorithm for SPR \eqref{SPR} with different values of sparsity parameter in the \textbf{noise-free setting}: accurate retrieval for sparsity parameter ranging from $100\%$ up to $500\%$ of the sparsity level.}
    \label{fig:SROP_stab_wrt_s_hat_noiseless}
\end{figure}
\begin{figure}[H]
	\centering
    \includegraphics[scale=.43]{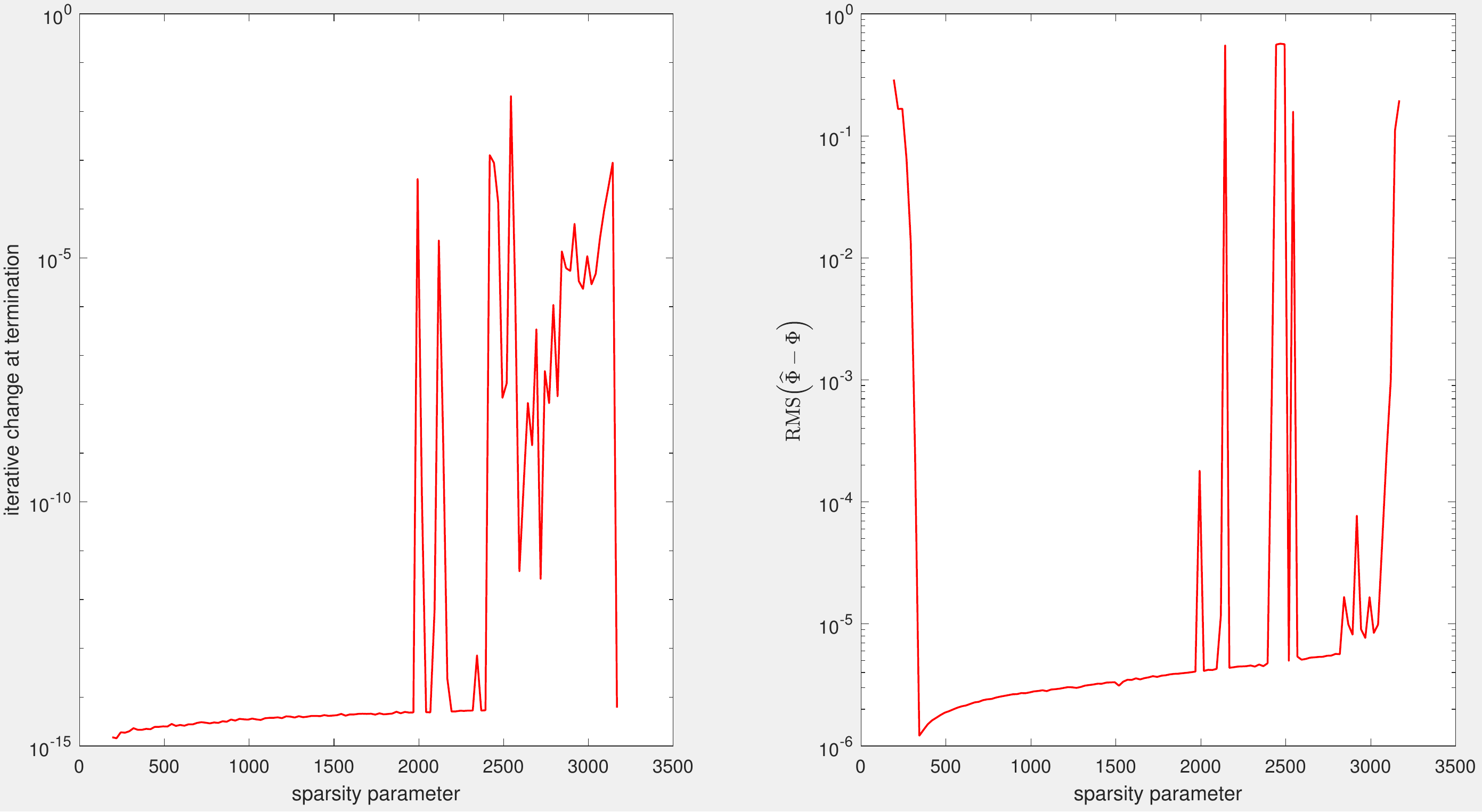}
    \caption{Performance of SROP algorithm for SPR \eqref{SPR} with different values of sparsity parameter in the \textbf{setting with Poisson noise}: accurate retrieval for sparsity parameter ranging from $100\%$ up to $500\%$ of the sparsity level.}
    \label{fig:SROP_stab_wrt_s_hat_noise}
\end{figure}

The performance of SROP algorithm for SPR \eqref{SPR} in the settings without noise and with Poisson noise is summarized in Figures \ref{fig:SROP_stab_wrt_s_hat_noiseless} and \ref{fig:SROP_stab_wrt_s_hat_noise}, respectively.
For sparsity parameter smaller than the sparsity level of $\Phi$ (e.g., $s<\|\Phi\|_0 = 319$), SROP algorithm fails to restore the phase as expected since the feasibility problem \eqref{auxi_FP} becomes inconsistent even in the noise-free case.
For sparsity parameter from the sparsity level of $\Phi$ up to $500\%$ of that value, the phase restoration by SROP algorithm is more or less perfect.
This clearly demonstrates that the sparsity regularization on phase for SPR \eqref{SPR} is very stable with respect to the sparsity parameter.

\section{Conclusion}\label{s:conclusion}

The phase retrieval problem with sparse phase constraint (SPR) has been analyzed in this paper for the first time.
Recall again that the considered problem is in essence different from the sparse signal recovery problem which is widely known and also well investigated in the literature.
We have developed a new sparsity regularization scheme (Section \ref{s:sparsity constraint}) which is then applied to design an efficient solution method (called SROP algorithm, Section \ref{s:SROP}) for the SPR problem.
Each iteration of SROP algorithm consists of two main steps: 1) an iteration of the Gerchberg-Saxton algorithm and 2) the sparsity regularization applied to the iteratively estimated phase to capture the sparsity property of SPR.
The latter step is mathematically a rotation on the underlying space $\mathbb{C}^{n\times n}$ but with direction varying in iterations.
We have proved in Section \ref{s:analysis} that this rotation step is indeed a metric projection on an auxiliary set which is independent of iterations.
This result is rather unexpected since the sparsity regularization applied to the phase on $(-\pi,\pi]^{n\times n}$ turns out to be a metric projection on the underlying complex Hilbert space $\mathbb{C}^{n\times n}$.
It is worth emphasizing that this equivalence occurs only for uniform amplitude modulation (Remark \ref{r:only for uniform}).
As a consequence, SROP algorithm is proved to be equivalent to the cyclic projections algorithm for solving a feasibility problem involving three auxiliary sets.
The convergence analysis of SROP algorithm is then obtained based on recent results for the cyclic projections algorithm for feasibility by analyzing regularity properties of the latter auxiliary sets (Theorem \ref{t:convergence}).
Our numerical results have proved clear empirical advantages of the new scheme of sparsity regularization on phase (Section \ref{s:numerical simulation}).
Overall, we have achieved the research goal of this paper which has been to design and analyze an efficient solution algorithm for solving the sparse phase retrieval problem \eqref{SPR} given only a single intensity PSF image.
Practical application of \emph{phase retrieval with sparse phase constraint} with respect to the orthogonal basis of Zernike polynomials will be investigated in a follow-up paper.




\end{document}